\documentclass[10pt,a4paper]{amsart}
\usepackage{amssymb,amsmath,amsfonts}

\headsep=1.2500cm \numberwithin{equation}{section}
\newtheorem{Theorem}{Theorem}[section]
\newtheorem{Proposition}[Theorem]{Proposition}
\newtheorem{cor}[Theorem]{Corollary}
\newtheorem{lemma}[Theorem]{Lemma}
\theoremstyle{remark}
\newtheorem{Definition}[Theorem]{Definition}
\newtheorem{Example}[Theorem]{Example}

\newtheorem{Remark}[Theorem]{Remark}

\begin{document}
\title{On approximately left $\phi$-biprojective Banach algebras}
\author{A. Sahami}
\address{Faculty of Mathematics and Computer Science,
Amirkabir University of Technology, 424 Hafez Avenue, 15914 Tehran,
Iran.} \email{amir.sahami@aut.ac.ir}


\begin{abstract}
In this paper, for a Banach algebra $A$, we introduced the new
notions of approximately left $\phi$-biprojective and approximately
left character biprojective, where $\phi$ is a non-zero
multiplicative linear functional on $A$. We show that for $SIN$
group $G$, Segal algebra $S(G)$ is approximately left
$\phi_{1}$-biprojective if and only if $G$ is amenable, where
$\phi_{1}$ is the augmentation character on $S(G)$. Also we showed
that the measure algebra $M(G)$ is approximately left character
biprojective if and only if $G$ is discrete and amenable. For a
Clifford semigroup $S$, we show that $\ell^{1}(S)$ is approximately
left character biprojective if and only if $\ell^{1}(S)$ is
pseudo-amenable. We study the hereditary property of these new
notions. Finally we give some examples among semigroup algebras and
Triangular Banach algebras to show the differences of these notions
and the classical ones.
\end{abstract}



\maketitle

\section{Introduction and Preliminaries}
The class  of amenable Banach algebras has been introduced by
Johnson. A Banach algebra $A$ is called amenable if for every
continuous derivation $D:A\rightarrow X^{*}$  there exists $x_{0}\in
X^{*}$ such that
$$D(a)=a\cdot x_{0}-x_{0}\cdot a\quad(a\in A).$$ He also showed that  $A$ is amenable if and only if
there exists a bounded net $(m_{\alpha})$ in $A\otimes_{p}A$ such
that $$a\cdot m_{\alpha}-m_{\alpha}\cdot a\rightarrow 0,\quad
\pi_{A}(m_{\alpha})a\rightarrow a\qquad(a\in A),$$ where
$\pi_{A}:A\otimes_{p}A\rightarrow A$ is given by $\pi_{A}(a\otimes
b)=ab$ for every $a,b\in A$, see \cite{Joh}. There is another
approach to study Banach algebra through the homological theory. Two
important notions of biflatness and biprojectivity for Banach
algebras have key role in homological theory. In fact a Banach
algebra $A$ is called biflat (biprojective), if there exists a
bounded $A$-bimodule morphism $\rho:A\rightarrow
(A\otimes_{p}A)^{**}$ ($\rho:A\rightarrow A\otimes_{p}A$) such that
$\pi_{A}^{**}\circ\rho$ is the canonical embedding of $A$ into
$A^{**}$ ($\rho$ is a right inverse for $\pi_{A}$), respectively see
\cite{hel}.  Note that a Banach algebra $A$ is amenable if and only
if $A$ is biflat and $A$ has a bounded approximate identity. In fact
the dual notion for amenability in Banach homology is biflatness. It
is known that for a locally compact group $G$, $L^{1}(G)$ is biflat
(biprojective) if and only if $G$ is amenable(compact),
respectively.

Recently a notion of amenability  related to a character has been
introduced in \cite{kan}. Indeed a Banach algebra $A$ is called left
$\phi$-amenable, if there exists a bounded net $(a_{\alpha})$ in $A$
such that $aa_{\alpha}-\phi(a)a_{\alpha}\rightarrow 0$ and
$\phi(a_{\alpha})\rightarrow 1$ for all $a\in A,$ where $\phi
\in\Delta(A).$ For a locally compact group $G$, the Fourier algebra
$A(G)$ is always left $\phi$-amenable. Also the group algebra
$L^{1}(G)$ is left $\phi$-amenable if and only if $G$ is amenable.
For the further information about this notion see \cite{san} and
\cite{alagh}. Motivated by these consideration, author with A.
Pourabbas introduced  notions of homological algebra theory related
to a character. A Banach algebra $A$ is called
$\phi$-biflat($\phi$-biprojective) if there exists a bounded
$A$-bimodule morphism  $$\rho:A\rightarrow
(A\otimes_{p}A)^{**}(\rho:A\rightarrow A\otimes_{p}A)$$ such that
$$\tilde{\phi}\circ\pi^{**}_{A}\circ\rho(a)=\phi(a)(\phi\circ\pi_{A}\circ\rho(a)=\phi(a))\quad
(a\in A),$$ respectively, where $\tilde{\phi}(F)=F(\phi)$ for all
$F\in A^{**}$. For a locally compact group $G$, we showed that Segal
algebra $S(G)$ is $\phi-$biflat(biprojective) if and only if $G$ is
amenable(compact). Also $A(G)$ is $\phi$-biprojective if and only if
$G$ is discrete, see \cite{sah rom} and \cite{sah1}.

Recently approximate versions of amenability and homological
properties of Banach algebras have been under more observations. In
\cite{zhang} Zhang introduced the notion of approximately
biprojective Banach algebras, that is, $A$ is approximately
biprojective if there exists a net of $A$-bimodule morphism
$\rho_{\alpha}:A\rightarrow A\otimes_{p}A$ such that
$$\pi_{A}\circ\rho_{\alpha}(a)\rightarrow a\quad(a\in A).$$ Author
with A. Pourabbas investigated approximate biprojectivity of some
semigroup algebras and some related Triangular Banach algebras see
\cite{sah col} and \cite{sah3}.
Approximate amenable Banach algebras
have been introduced by Ghahramani and Loy. Indeed a Banach algebra
$A$ is approximate amenable if for every continuous derivation
$D:A\rightarrow X^{*}$, there exists a net $(x_{\alpha})$ in $X^{*}$
such that
$$D(a)=\lim_{\alpha}a\cdot x_{\alpha}-x_{\alpha}\cdot a \quad(a\in
A).$$ Other extensions of amenability are pseudo-amenability and
pseudo-contractibility. A Banach algebra $A$ is pseudo-amenable
(pseudo-contractible) if there exists a not necessarily bounded net
$(m_{\alpha})$ in $A\otimes_{p}A$ such that
$$a\cdot m_{\alpha}-m_{\alpha}\cdot a\rightarrow 0,\quad(a\cdot m_{\alpha}=m_{\alpha}\cdot a),\qquad \pi_{A}(m_{\alpha})a\rightarrow a\quad (a\in A),$$
respectively. For more information about these concepts the reader
referred to \cite{ghah pse}, \cite{gen 1} and \cite{gen 2}.
Motivated by these considerations, in \cite{agha}  the approximate
notions of amenability have been introduced and studied. A Banach
algebra $A$ is called approximately left $\phi$-amenable if there
exists a (not necessarily bounded) net $(a_{\alpha})$ in $A$ such
that $aa_{\alpha}-\phi(a)a_{\alpha}\rightarrow 0$ and
$\phi(a_{\alpha})\rightarrow 1$ for all $a\in A.$ Also $A$ is
approximately character amenable, if $A$ is approximately left
$\phi$-amenable for all $\phi\in\Delta(A)\cup\{0\}.$ They showed
that ${L^{1}(G)}^{**}$ is character amenable if and only if $G$ is
discrete and amenable. Also they showed that $M(G)$  is character
amenable if and only if $G$ is discrete and amenable.

In this paper, we extend the notions of $\phi-$biflatness and
$\phi-$biprojectivity. We  give an approximate notion of homological
algebra related to approximate left $\phi$-amenability. Here the
definition of our new notion:
\begin{Definition}\label{def}
Let $A$ be a Banach algebra and $\phi\in\Delta(A)$. $A$ is called
approximately left $\phi$-biprojective if there exists a net of
bounded linear maps from $A$ into $A\otimes_{p}A$, say
$(\rho_{\alpha})_{\alpha\in I}$, such that
\begin{enumerate}
\item [(i)] $a\cdot \rho_{\alpha}(x)-\rho_{\alpha}(ax)\xrightarrow{||\cdot||} 0$;
\item [(ii)] $\rho_{\alpha}(xa)-\phi(a)\rho_{\alpha}(x)\xrightarrow{||\cdot||} 0$;
\item [(iii)] $\phi\circ\pi_{A}\circ\rho_{\alpha}(x)-\phi(x)\rightarrow
0$,
\end{enumerate}
for every $a,x\in A$. We say that $A$ is approximately left
character biprojective if $A$ is approximately left
$\phi$-biprojective for all $\phi\in\Delta(A)$.
\end{Definition}
We show that approximate left $\phi$-amenability gives approximate
left $\phi$-biprojectivity. We study the hereditary properties of
this new notion. We show that
 for $SIN$ group $G$, Segal
algebra $S(G)$ is approximately left $\phi_{1}$-biprojective if and
only if $G$ is amenable, where $\phi_{1}$ is augmentation character
on $S(G)$. Also we showed that the measure algebra $M(G)$ is
approximately left character  biprojective if and only if $G$ is
discrete and amenable. We give some Banach algebra among Triangular
Banach algebras  which is never approximately left
$\phi$-biprojective. Also we give some examples which reveal the
differences our new notions and the classical ones.

We remark some standard notations and definitions that we shall need
in this paper. Let $A$ be a Banach algebra. Throughout this paper
the character space of $A$ is denoted by $\Delta(A)$, that is, all
non-zero multiplicative linear functionals on $A$.  Let $A$  be a
Banach algebra.  The projective tensor product
 $A\otimes_{p}A$ is a Banach $A$-bimodule via the following actions
$$a\cdot(b\otimes c)=ab\otimes c,~~~(b\otimes c)\cdot a=b\otimes
ca\hspace{.5cm}(a, b, c\in A).$$ Let $A$ and $B$ be Banach algebras
and $\phi\in\Delta(A)$ and $\psi\in\Delta(B)$. We denote
$\phi\otimes \psi$ for a map defined by $\phi\otimes \psi(a\otimes
b)=\phi(a)\psi(b)$ for all $a\in A$ and $b\in B.$ It is easy to see
that $\phi\otimes \psi\in\Delta(A\otimes_{p}B).$

Let $A$ and $B$ be a Banach algebras and let $X$ be a Banach
$A,B$-module, that is, $X$ is a Banach space, a left $A$-module and
a right $B$-module with the compatible module action that satisfies
$(a\cdot x)\cdot b=a\cdot (x\cdot b)$ and $||a\cdot x\cdot b||\leq
||a||||x||||b||$ for every $a\in A, x\in X, b\in B$.  With the usual
matrix operation and $||\left(\begin{array}{cc} a&x\\
0&b\\
\end{array}
\right)||=||a||+||x||+||b||$,
$T=\left(\begin{array}{cc} A&X\\
0&B\\
\end{array}
\right)$ becomes a Banach algebra which is called Triangular Banach
algebra.
Let $\phi\in\Delta(B)$. We define a character  $\psi_{\phi}\in\Delta(T)$  via  $\psi_{\phi}\left(\begin{array}{cc} a&x\\
0&b\\
\end{array}
\right)=\phi(b)$ for every $a\in A$, $b\in B$ and $x\in X$.

For a locally compact group $G$, $M(G)$ is denoted for measure
algebra and $A(G)$ is denoted for Fourier algebra.
\section{Approximate left $\phi$-biprojectivity}
In this section we study the general properties of approximately
left $\phi$-biprojective Banach algebras.
\begin{Proposition}\label{inner cont}
Let $A$ be a  Banach algebra with $\phi\in\Delta(A)$. Suppose that
$A$ is an approximately left $\phi$-biprojective Banach algebra
which has an element $a_{0}$ such that $aa_{0}=a_{0}a$ and
$\phi(a_{0})=1$. Then $A$ is approximately left $\phi$-amenable.
\end{Proposition}
\begin{proof}
Let $(\rho_{\alpha})_{\alpha\in I}$ be as in  Definition \ref{def}.
Let $a_{0}$ be an element in $A$ such that $aa_{0}=a_{0}a$ and
$\phi(a_{0})=1$. Set $n_{\alpha}=\rho_{\alpha}(a_{0})$. It is clear
that $(n_{\alpha})$ is a net in $A\otimes_{p}A$ such that
\begin{equation*}
\begin{split}
a\cdot n_{\alpha}-\phi(a)n_{\alpha}&=a\cdot
\rho_{\alpha}(a_{0})-\phi(a)\rho_{\alpha}(a_{0})\\
&=a\cdot
\rho_{\alpha}(a_{0})-\rho_{\alpha}(aa_{0})+\rho_{\alpha}(aa_{0})-\rho_{\alpha}(a_{0}a)+\rho_{\alpha}(a_{0}a)-\phi(a)\rho_{\alpha}(a_{0})\\
&\rightarrow 0
\end{split}
\end{equation*}
for every $a\in A. $ Also we have
$$\phi\circ\pi_{A}(n_{\alpha})-1=\phi\circ\pi_{A}\circ\rho_{\alpha}(a_{0})-\phi(a_{0})\rightarrow0.$$
Define $T:A\otimes_{p}A\rightarrow A$ by $T(a\otimes b)=\phi(b)T(a)$
for each $a\in A$ and $b\in B.$ It is clear that $T$ is a bounded
linear map which satisfies $$T(a\cdot x)=aT(x),\quad T(x\cdot
a)=\phi(T(x)),\quad \phi\circ T=\phi\circ\pi_{A},\quad (a\in A,x\in
A\otimes_{p}A).$$ Set $m_{\alpha}=T(n_{\alpha})$. One can show that
$$aT(n_{\alpha})-\phi(a)T(n_{\alpha})=T(a\cdot n_{\alpha}-\phi(a)n_{\alpha})\rightarrow 0,\quad (a\in A)$$
and
$$\phi(m_{\alpha})=\phi\circ T(n_{\alpha})=\phi\circ\pi_{A}(n_{\alpha})\rightarrow 1.$$
Thus  $A$ is approximately left $\phi$-amenable.
\end{proof}
\begin{Proposition}\label{app biproj}
Let $A$ be a Banach algebra with $\phi\in\Delta(A)$. If $A$ is
approximately biprojective, then $A$ is approximately left
$\phi$-biprojective.
\end{Proposition}
\begin{proof}
Since $A$ is approximately biprojective, there exists a net of
$A$-bimodule morphism $\rho_{\alpha}:A\rightarrow A\otimes_{p}A$
such that
$$\pi_{A}\circ\rho_{\alpha}(a)\rightarrow a\quad(a\in A).$$ Pick $a_{0}\in
A$ such that $\phi(a_{0})=1.$ Set $T:A\otimes_{p}A\rightarrow
A\otimes_{p}A$ which defined by $T(a\otimes b)=\phi(a)a_{0}\otimes
b$ for each $a,b\in A.$ It is easy to see that
$$\phi(x)T(a\otimes b)=\phi(x)\phi(a)a_{0}\otimes b=\phi(xa)a_{0}\otimes b=T(x\cdot a\otimes b)$$
and
$$\phi\circ\pi_{A}\circ T(a\otimes b)=\phi(\phi(a)a_{0}b)=\phi(ab)=\phi\circ\pi_{A}(a\otimes b)$$
for each $a,b,x\in A.$ Using these facts one can see that
$(T\circ\rho_{\alpha})_{\alpha}$ satisfies the conditions in
Definition \ref{def}. So  $A$ is approximately left
$\phi$-biprojective.
\end{proof}
There exists a Banach algebra  which is never approximately left
$\phi$-biprojective.
\begin{Example}\label{ex1}
Consider a Triangular Banach algebra $T=\left(\begin{array}{cc} \mathbb{C}&\mathbb{C}\\
0&\mathbb{C}\\
\end{array}
\right).$ Define $\phi\in\Delta(T)$ by $\phi(\left(\begin{array}{cc} a&b\\
0&c\\
\end{array}
\right))=c $ for all $a,b,c\in\mathbb{C}.$ We claim that $T$ is not
approximately $\phi$-biprojective. To see this we go toward a
contradiction and assume that $T$ is approximately left
$\phi$-biprojective. Since $T$ is unital, by Proposition \ref{inner
cont} $T$ is
approximately left $\phi$-amenable. Set $I=\left(\begin{array}{cc} 0&\mathbb{C}\\
0&\mathbb{C}\\
\end{array}
\right).$ It is easy to see that $\phi|_{I}\neq 0$ then by
\cite[Proposition 5.1]{saha} $I$ is approximately left
$\phi$-amenable. Thus there exists  a net $(i_{\alpha})$ in $I$ such
that $$ii_{\alpha}-\phi(i)i_{\alpha}\rightarrow 0,\quad
\phi(i_{\alpha})\rightarrow 1,\quad(i\in I).$$ Hence there exist net
$(a_{\alpha})$ and $(b_{\alpha})$ in $\mathbb{C}$ such that
$i_{\alpha}=\left(\begin{array}{cc} 0&a_{\alpha}\\
0&b_{\alpha}\\
\end{array}
\right).$ So for each  $i=\left(\begin{array}{cc} 0&a\\
0&b\\
\end{array}
\right)$ in $I$, we have $$\left(\begin{array}{cc} 0&a\\
0&b\\
\end{array}
\right)\left(\begin{array}{cc} 0&a_{\alpha}\\
0&b_{\alpha}\\
\end{array}
\right)-b\left(\begin{array}{cc} 0&a_{\alpha}\\
0&b_{\alpha}\\
\end{array}
\right)\rightarrow 0.$$ Then $ab_{\alpha}-ba_{\alpha}\rightarrow 0,$
for each $a,b\in \mathbb{C}$. Since $b_{\alpha}\rightarrow 1,$
taking $a=1$ and $b=0$, gives a contradiction.
\end{Example}
We remind that by \cite[Proposition 2.7]{agha}, $A$ is approximately
left $\phi$-amenable if and only if there exists a net
$(m_{\alpha})$ in $(A\otimes_{p}A)^{**}$ such that $$a\cdot
m_{\alpha}-\phi(a)m_{\alpha}\rightarrow 0,\quad
\tilde{\phi}\circ\pi^{**}_{A}(m_{\alpha})\rightarrow 1,$$ for all
$a\in A.$
\begin{Proposition}\label{app left amen}
Let $A$ be a Banach algebra and $\phi\in\Delta(A).$ If $A$ is
approximately left $\phi$-amenable, then $A$ is approximately left
$\phi$-biprojective.
\end{Proposition}
\begin{proof}
Let $A$ be approximately left $\phi$-amenable. Then there exists a
net $m_{\alpha}$ in $(A\otimes_{p}A)^{**}$ such that $a\cdot
m_{\alpha}-\phi(a)m_{\alpha}\rightarrow 0$ and
$\tilde{\phi}\circ\pi^{**}(m_{\alpha})=1,$ for each $a\in A,$ see
\cite[Proposition 2.7]{agha}. Take $\epsilon>0,$ $F\subseteq A$ and
$\Lambda\subseteq (A\otimes_{p}A)^{*}$ arbitrary finite subsets.
Then we have
$$||a\cdot
m_{\alpha}-\phi(a)m_{\alpha}||<\epsilon,\quad
|\tilde{\phi}\circ\pi_{A}^{**}(m_{\alpha})-1|<\epsilon,\quad (a\in
F).$$ It is well-known that for each $\alpha$, there exists a net
$(n^{\alpha}_{\beta})_{\beta}$ in $A\otimes_{p}A$ such that
$n^{\alpha}_{\beta}\xrightarrow{w^{*}}m_{\alpha}$. Since
$\pi^{**}_{A}$ is a $w^{*}$-continuous map, then
$$\pi_{A}(n^{\alpha}_{\beta})=\pi_{A}^{**}(n^{\alpha}_{\beta})\xrightarrow{w^{*}}\pi_{A}^{**}(m_{\alpha}).$$
Thus we have $$|a\cdot
n^{\alpha}_{\beta}(f)-am_{\alpha}(f)|<\frac{\epsilon}{K_{0}},
|\phi(a)n^{\alpha}_{\beta}(f)-\phi(a)m_{\alpha}(f)|<\frac{\epsilon}{K_{0}},
|\phi\circ\pi_{A}(n^{\alpha}_{\beta})-\tilde{\phi}\circ\pi^{**}(m_{\alpha})|<\epsilon,\quad
(a\in F, f\in\Lambda),$$ where $K_{0}=sup\{||f|||f\in \Lambda\}$.
Since $a\cdot m_{\alpha}-\phi(a)m_{\alpha}\rightarrow 0$ and
$\tilde{\phi}\circ\pi^{**}(m_{\alpha})=1,$ we can find
$\beta=\beta(F,\Lambda,\epsilon)$ such that
$$|a\cdot n^{\alpha}_{\beta}(f)-\phi(a)n^{\alpha}_{\beta}(f)|<c\frac{\epsilon}{K_{0}},\quad |\phi\circ\pi_{A}(n^{\alpha}_{\beta})-1|<\epsilon,\quad (a\in F,f\in \Lambda)$$
for some $c\in \mathbb{R}$. Using Mazur's lemma, we have a net
$(n_{(F,\Lambda,\epsilon)})$ in $A\otimes_{p}A$ such that $$||a\cdot
n_{(F,\Lambda,\epsilon)}-\phi(a)n_{(F,\Lambda,\epsilon)}||\rightarrow
0,\quad | \phi\circ\pi_{A}(n_{(F,\Lambda,\epsilon)})-1|\rightarrow
0,\quad (a\in F).$$ Define $\rho_{(F,\Lambda,\epsilon)}:A\rightarrow
A\otimes_{p}A$ by $\rho_{(F,\Lambda,\epsilon)}(a)=a\cdot
n_{(F,\Lambda,\epsilon)} $ for each $a\in A.$ It is clear that
$\rho_{(F,\Lambda,\epsilon)}(ab)=a\cdot
\rho_{(F,\Lambda,\epsilon)}(b)$ for each $a,b\in A.$
\begin{equation}
\begin{split}
||\rho_{(F,\Lambda,\epsilon)}(ab)-\phi(b)\rho_{(F,\Lambda,\epsilon)}(a)||&=||ab\cdot
n_{(F,\Lambda,\epsilon)}-\phi(b)a\cdot n_{(F,\Lambda,\epsilon)}||\\
&\leq ||a||||b\cdot n_{(F,\Lambda,\epsilon)}-\phi(b)
n_{(F,\Lambda,\epsilon)}||\rightarrow 0,
\end{split}
\end{equation}
for each $a,b\in A.$ Also
\begin{equation}
\begin{split}
|\phi\circ\pi_{A}\circ\rho_{(F,\Lambda,\epsilon)}(a)-\phi(a)|=|\phi\circ\pi_{A}(a\cdot
n_{(F,\Lambda,\epsilon)})-\phi(a)|&=|\phi(a)||\phi\circ\pi_{A}(n_{(F,\Lambda,\epsilon)})-1|\\
&\rightarrow 0,
\end{split}
\end{equation}
for each $a\in A-\ker\phi$. It is easy to see that
$\phi\circ\pi_{A}\circ\rho_{(F,\Lambda,\epsilon)}(a)=\phi(a)$ for
each $a\in\ker\phi$.
\end{proof}
\begin{Remark}\label{rem1}
Let $A$ be a Banach algebra and $\phi\in\Delta(A)$. Using the
argument of previous Theorem one can see that if $A$ is either
pseudo-amenable or approximately amenable, then $A$ is approximately
left $\phi$-biprojective.
\end{Remark}
\begin{Theorem}\label{phi biflat}
Let $A$ be  a Banach algebra and $\phi\in\Delta(A).$ If $A$ is
$\phi$-biflat, then $A$ is approximately left $\phi$-biprojective.
\end{Theorem}
\begin{proof}
Since $A$ is $\phi$-biflat then there exists a bounded $A$-bimodule
morphism  $\rho:A\rightarrow (A\otimes_{p}A)^{**}$ such that
$\tilde{\phi}\circ\pi^{**}_{A}\circ \rho(a)=\phi(a)$ for each $ a\in
A.$ There exists a net $(\rho_{\alpha})$ in $B(A,A\otimes_{p}A)$
(the set of bounded linear maps from $A$ into $A\otimes_{p}A$) such
that $\rho_{\alpha}$ converges to $\rho$ in the $weak-star$
$operator$ $topology$. Since $\pi^{**}_{A}$ is a $w^{*}$-continuous
map, for each $a\in A$ we have
$$\pi_{A}\circ \rho_{\alpha}(a)=\pi^{**}_{A}\circ
\rho_{\alpha}(a)\xrightarrow{w^{*}}\pi_{A}^{**}\circ\rho(a),$$ so
$$\phi\circ \pi_{A}\circ \rho_{\alpha}(a)\rightarrow
\tilde{\phi}\circ\pi^{**}_{A}\circ \rho(a).$$ Let $\epsilon>0$ and
take $F=\{a_{1},a_{2},...,a_{r}\}$ and $G=\{x_{1},x_{2},...,x_{r}\}$
 arbitrary finite subsets of $A$. Define
\begin{equation}
\begin{split}
M=\{(a_{1}\cdot T(x_{1})-T(a_{1}x_{1}),a_{2}&\cdot T(x_{2})-T(a_{2}x_{2}),...,a_{r}\cdot T(x_{r})-T(a_{r}x_{r}),\\
&\phi\circ\pi_{A}\circ T(x_{i})-\phi(x_{i}))|T\in
B(A,A\otimes_{p}A)\}_{i=1,...,r}\subseteq
\prod^{r}_{i=1}(A\otimes_{p}A)\oplus_{1}\mathbb{C}.
\end{split}
\end{equation}
It is clear that $M$ is a convex set and $(0,0,...,0)$ belongs to
$\overline{M}^{w}$. Using Mazur's Lemma $(0,0,...,0)\in
\overline{M}^{w}=\overline{M}^{||\cdot||}$. Then we can find an
element  $\theta_{(F,G,\epsilon)}$ in $B(A,A\otimes_{p}A)$ such that
$$||a_{i}\cdot \theta_{(F,G,\epsilon)}(b_{i})-\theta_{(F,G,\epsilon)}(a_{i}b_{i})||<\epsilon,\quad ||\theta_{(F,G,\epsilon)}(a_{i}b_{i})-\theta_{(F,G,\epsilon)}(a_{i})\cdot b_{i}||<\epsilon$$
and
$$|\phi\circ\pi_{A}\circ\theta_{(F,G,\epsilon)}(a_{i})-\phi(a_{i})|<\epsilon,$$
for each $i\in\{1,2,...,r\}$. Hence the net
$(\theta_{(F,G,\epsilon)})_{(F,G,\epsilon)}$ satisfies
$$a\cdot \theta_{(F,G,\epsilon)}(b)-\theta_{(F,G,\epsilon)}(ab)\rightarrow 0,\quad \theta_{(F,G,\epsilon)}(ab)-\theta_{(F,G,\epsilon)}(a)\cdot b\rightarrow 0$$
and
$\phi\circ\pi_{A}\circ\theta_{(F,G,\epsilon)}(a)-\phi(a)\rightarrow
0$ for each $a,b \in A.$ Set $T$ the same map as in the proof of
Proposition \ref{app biproj}. It is easy to see that $(T\circ
\theta_{(F,G,\epsilon)})_{(F,G,\epsilon)}$ satisfies the conditions
in Definition \ref{def}. So $A$ is approximately left
$\phi$-biprojective.
\end{proof}
We have to remind that every biflat  Banach algebra $A$ with
$\phi\in\Delta(A)$ is $\phi$-biflat. Then Using previous Theorem, we
have the following corollary:
\begin{cor}
Suppose that $A$ is a biflat Banach algebra with $\phi\in\Delta(A)$.
Then $A$ is approximately left $\phi-$biprojective.
\end{cor}
\begin{Proposition}\label{ideal}
Let $A$ be a Banach algebra and $\phi\in\Delta(A)$. Suppose that $I$
is closed ideal of $A$ which $\phi|_{I}\neq 0.$ If $A$ is
approximately left $\phi$-biprojective, then $I$ is approximately
left $\phi$-biprojective.
\end{Proposition}
\begin{proof}
Let $(\rho_{\alpha})_{\alpha}$ be a net of maps which satisfies
Definition \ref{def}. Take $i_{0}$ in $I$ such that $\phi(i_{0}
)=1.$ Define $T:A\otimes_{p}A\rightarrow I\otimes_{p}I$ by
$T(a\otimes b)=ai_{0}\otimes i_{0}b$ for each $a,b\in A.$ It is easy
to see that $T$ is a bounded linear map. Set
$\eta_{\alpha}=T\circ\rho_{\alpha}|_{I}:I\rightarrow I\otimes_{p}I.$
Then we have
$$i\cdot \eta_{\alpha}(j)-\eta_{\alpha}(ij)=T(i\cdot \rho_{\alpha}(j)-\rho_{\alpha}(ij))\rightarrow 0$$
and
$$\eta_{\alpha}(ij)-\phi(j)\eta_{\alpha}(i)=T(\rho_{\alpha}(ij)-\phi(j)\rho_{\alpha}(i))\rightarrow 0$$
also
$$\phi\circ\pi_{I}\circ\eta_{\alpha}(i)-\phi(i)=\phi\circ\pi_{I}\circ T\circ\rho_{\alpha}(i)-\phi(i)=\phi\circ\pi_{A}\circ\rho_{\alpha}(i)-\phi(i)\rightarrow 0,$$
for each $i,j\in I.$
\end{proof}
\begin{Theorem}
Let $A$ and $B$ be Banach algebras and $\phi\in\Delta(A)$ and
$\psi\in\Delta(B)$. Suppose that $A$ is unital and $B$ has an
idempotent  $x_{0}$  such that $\psi(x_{0})=1.$ If $A\otimes_{p}B$
is approximately left $\phi\otimes\psi$-biprojective, then $A$ is
approximately left $\phi$-biprojective.
\end{Theorem}
\begin{proof}
Let $(\rho_{\alpha}):A\otimes_{p}B\rightarrow
(A\otimes_{p}B)\otimes_{p}(A\otimes_{p}B)$ be a net of continuous
maps such that $$x\cdot\rho_{\alpha}(y)-\rho_{\alpha}(xy)\rightarrow
0,\quad\rho_{\alpha}(xy)-\phi(x)\rho_{\alpha}(y)\rightarrow 0$$ and
$$\phi\otimes\psi \circ
\pi_{A\otimes_{p}B}(x)-\phi\otimes\psi(x)\rightarrow 0,$$ for each
$x,y\in A\otimes_{p}B.$ Note that $A\otimes_{p}B$ with the following
actions becomes a Banach $A$-bimodule:$$a_{1}\cdot(a_{2}\otimes
b)=a_{1}a_{2}\otimes b,\quad (a_{2}\otimes b)\cdot
a_{1}=a_{2}a_{1}\otimes b,\quad (a_{1},a_{2}\in A,b\in B).$$
Consider
\begin{equation*}
\begin{split}
\rho_{\alpha}(a_{1}a_{2}\otimes
x_{0})-a_{1}\cdot\rho_{\alpha}(a_{2}\otimes x_{0})&=
\rho_{\alpha}((a_{1}\otimes x_{0})(a_{2}\otimes
x_{0}))-a_{1}\cdot\rho_{\alpha}(a_{2}\otimes x_{0})\\
&=\rho_{\alpha}((a_{1}\otimes x_{0})(a_{2}\otimes
x_{0}))-(a_{1}\otimes x_{0})\cdot \rho_{\alpha}(a_{2}\otimes
x_{0})+\\
&(a_{1}\otimes x_{0})\cdot \rho_{\alpha}(a_{2}\otimes
x_{0})-a_{1}\cdot\rho_{\alpha}(a_{2}\otimes x_{0})\\
&=\rho_{\alpha}((a_{1}\otimes x_{0})(a_{2}\otimes
x_{0}))-(a_{1}\otimes x_{0})\cdot \rho_{\alpha}(a_{2}\otimes
x_{0})+\\
&(a_{1}\cdot (e\otimes x_{0}))\cdot \rho_{\alpha}(a_{2}\otimes
x_{0})-a_{1}\cdot\rho_{\alpha}(a_{2}\otimes x_{0})\\
&=\rho_{\alpha}((a_{1}\otimes x_{0})(a_{2}\otimes
x_{0}))-(a_{1}\otimes x_{0})\cdot \rho_{\alpha}(a_{2}\otimes
x_{0})+\\
&(a_{1}\cdot (e\otimes x_{0}))\cdot \rho_{\alpha}(a_{2}\otimes
x_{0})-a_{1}\cdot\rho_{\alpha}(ea_{2}\otimes x_{0}x_{0})+\\
&a_{1}\cdot\rho_{\alpha}(ea_{2}\otimes
x_{0}x_{0})-a_{1}\cdot\rho_{\alpha}(a_{2}\otimes x_{0})\rightarrow 0
\end{split}
\end{equation*}
and
\begin{equation*}
\begin{split}
\rho_{\alpha}(a_{1}a_{2}\otimes
x_{0})-\phi(a_{2})\rho_{\alpha}(a_{1}\otimes
x_{0})&=\rho_{\alpha}((a_{1}\otimes x_{0})(a_{2}\otimes
x_{0}))-\phi(a_{2})\rho_{\alpha}(a_{1}\otimes x_{0})\\
&=\rho_{\alpha}((a_{1}\otimes x_{0})(a_{2}\otimes
x_{0}))-\phi\otimes\psi(a_{2}\otimes
x_{0})\rho_{\alpha}(a_{1}\otimes x_{0})+\\
&\phi\otimes\psi(a_{2}\otimes x_{0})\rho_{\alpha}(a_{1}\otimes
x_{0})-\phi(a_{2})\rho_{\alpha}(a_{1}\otimes x_{0})\rightarrow 0,
\end{split}
\end{equation*}
for each $a_{1},a_{2}\in A.$ Define
$T:(A\otimes_{p}B)\otimes_{p}(A\otimes_{p}B)\rightarrow
A\otimes_{p}A$ by $T((a\otimes b)\otimes (c\otimes
d))=\psi(bd)a\otimes c,$ for each $a,c\in A, b,d \in B.$ One can see
that $T$ is a bounded linear operator and $\pi_{A}\circ T=(id\otimes
\psi)\circ\pi_{A\otimes_{p}B},$ where $id\otimes \psi(a\otimes
b)=\psi(b)a$ for all $a\in A,b\in B.$ Set
$\eta_{\alpha}(a)=T\circ\rho(a\otimes x_{0})$. It is easy to see
that for each $\alpha$, $\eta_{\alpha}:A\rightarrow A\otimes_{p}A$
is a continuous map which satisfies $$a\cdot
\eta_{\alpha}(b)-\eta_{\alpha}(ab)\rightarrow 0,\quad
\eta_{\alpha}(ab)-\phi(b)\eta_{\alpha}(a)\rightarrow 0,\quad (a,b\in
A).$$ Also we have
$$\phi\circ\pi_{A}\circ\eta_{\alpha}(a)=\phi\circ\pi_{A}\circ T\circ\rho_{\alpha}(a\otimes x_{0})=\phi\circ (id\otimes \psi)\circ \pi_{A\otimes_{p}B}
\circ \rho_{\alpha}(a\otimes x_{0})\rightarrow \phi(a),$$ for each
$a\in A.$ Hence $A$ is approximately left $\phi$-biprojective.
\end{proof}
\section{Application to Banach algebras associated with a locally compact group}
Let $G$ be a locally compact group and let $\widehat{G}$ be its dual
group, which consists of all non-zero continuous homomorphism
$\zeta:G\rightarrow \mathbb{T}$. It is well-known that
$\Delta(L^{1}(G))=\{\phi_{\zeta}:\zeta \in \widehat{G}\}$, where
$\phi_{\zeta}(f)=\int_{G}\overline{\zeta(x)}f(x)dx$ and $dx$ is a
left Haar measure on $G$, for more details, see \cite[Theorem
23.7]{hew}.

The map $\phi_{1}:L^{1}(G)\rightarrow \mathbb{C}$ which is specified
by
$$\phi_{1}(f)=\int_{G} f(x)dx$$ is called augmentation character. We
know that augmentation character induce a character on $S(G)$ which
we denote it by $\phi_{1}$ again, see \cite{alagh}.

We recall that, for a locally compact group $G$, a linear subspace
$S(G)$ of  $L^{1}(G)$ is said to be a Segal algebra on $G$ if it
satisfies the following properties:
\begin{enumerate}
\item [(i)] $S(G)$ is a dense left ideal  in $L^{1}(G)$;
\item [(ii)]  $S(G)$ with respect to some norm $||\cdot||_{S(G)}$ is
a Banach space and $|| f||_{L^{1}(G)}\leq|| f||_{S(G)}$;
\item [(iii)] For $f\in S(G)$ and $y\in G$, $L_{y}f\in S(G)$
and the map $y\mapsto \delta_{y}\ast f$ is continuous. Also
$||\delta_{y}\ast f||_{S(G)}=|| f||_{S(G)}$, for $f\in S(G)$ and
$y\in G$.
\end{enumerate}
For more information about this algebras see \cite{rei}.
\begin{Theorem}\label{segal}
Let $G$ be a locally compact $SIN$-group. If $S(G)$ is approximately
$\phi_{1}$-biprojective, then $G$ is amenable.
\end{Theorem}
\begin{proof}
Using the main result of \cite{kot}, $G$ is a $SIN$ group if and
only if $S(G)$ has a central approximate identity. Then we have an
element $a_{0}\in S(G)$ such that $aa_{0}=a_{0}a$ and
$\phi_{1}(a_{0})=1$ for each $a\in S(G).$ Applying Proposition
\ref{inner cont}, approximate left $\phi_{1}$-biprojectivity of
$S(G)$ implies that $S(G)$ is approximately left
$\phi_{1}$-amenable. We can find a net $(m_{\alpha})$ in $S(G)$ such
that
$$||am_{\alpha}-\phi_{1}(a)m_{\alpha}||_{S(G)}\rightarrow 0,\quad \phi_{1}(m_{\alpha})\rightarrow 1\quad(a\in S(G)).$$
Since $||\cdot||_{L^{1}(G)}\leq ||\cdot||_{S(G)}$, then
$$||am_{\alpha}-\phi_{1}(a)m_{\alpha}||_{L^{1}(G)}\rightarrow 0,\quad \phi_{1}(m_{\alpha})\rightarrow 1\quad(a\in S(G)).$$
Let $f$ be an element of $S(G)$ such that $\phi_{1}(f)=1.$ Define
$f_{\alpha}=fm_{\alpha}$. For each $y\in G$ we have
$$\phi_{1}(\delta_{y}f)=\int_{G}\delta_{y} f(x)dx=\int_{G}f(y^{-1}x)dx=\int_{G}f(x)dx=\phi_{1}(f),$$
where $\delta_{y}$ denotes the point mass at $\{y\}$. We have
\begin{equation}
\begin{split}
||\delta_{y}f_{\alpha}-f_{\alpha}||_{L^{1}(G)}&=||(\delta_{y}f)m_{\alpha}-fm_{\alpha}||_{L^{1}(G)}\\
&\leq
||(\delta_{y}f)m_{\alpha}-m_{\alpha}||_{L^{1}(G)}+||m_{\alpha}-fm_{\alpha}||_{L^{1}(G)}\\
&\leq
||(\delta_{y}f)m_{\alpha}-\phi_{1}(\delta_{y}f)m_{\alpha}||_{L^{1}(G)}+||\phi_{1}(\delta_{y}f)m_{\alpha}-m_{\alpha}||_{L^{1}(G)}\\
&+||m_{\alpha}-\phi_{1}(f)m_{\alpha}
||_{L^{1}(G)}+||\phi_{1}(f)m_{\alpha}-fm_{\alpha}||_{L^{1}(G)}\\
&\rightarrow 0.
\end{split}
\end{equation}
On the other hand
$$\phi_{1}(f_{\alpha})=\phi_{1}(fm_{\alpha})=\phi_{1}(f)\phi_{1}(m_{\alpha})\rightarrow 1.$$
Since $|\phi_{1}(f_{\alpha})|\leq ||f_{\alpha}||_{L^{1}(G)}$, then
$(f_{\alpha})$ can stays away from $0$. Without loss of generality
we assume that $||f_{\alpha}||_{L^{1}(G)}\geq \frac{1}{2}.$ Define
$g_{\alpha}=\frac{|f_{\alpha}|}{||f_{\alpha}||_{L^{1}(G)}}$. It is
clear that $(g_{\alpha})$ is a bounded net in $L^{1}(G).$ Consider
$$||\delta_{y}g_{\alpha}-g_{\alpha}||_{L^{1}(G)}\leq 2||\delta_{y}|f_{\alpha}-|f_{\alpha}||||_{L^{1}(G)}\leq 2||\delta_{y}f_{\alpha}-f_{\alpha}||_{L^{1}(G)}
\rightarrow 0.$$ Now by  \cite[Exercise 1.1.6]{run}, $G$ is
amenable.
\end{proof}
We give a Banach algebra related to a locally compact group which is
never approximately left $\phi$-biprojective.
\begin{Example}
Let $G$ be a locally compact group and let $A(G)$ be the Fourier
algebra
with respect to $G$. Let  $T=\left(\begin{array}{cc} A(G)&A(G)\\
0&A(G)\\
\end{array}
\right)$ be a Triangular Banach algebra related to $A(G).$ Suppose
that $\phi\in\Delta(A(G))$. Define $\psi_{\phi}(\left(\begin{array}{cc} a&b\\
0&c\\
\end{array}
\right))=\phi(c)$ for all $a,b,c\in A(G)$. It is easy to see that
$\psi_{\phi}\in\Delta(T).$ Note that $A(G)$ is a commutative Banach
algebra, hence there exists $a_{0}\in A(G)$ such that
$aa_{0}=a_{0}a$ and $\phi(a_{0})=1$ for each $a\in A(G)$. Set $t_{0}=\left(\begin{array}{cc} a_{0}&0\\
0&a_{0}\\
\end{array}
\right)$, clearly $tt_{0}=t_{0}t, \psi_{\phi}(t_{0})=1,$ for every
$t\in T.$ Using Proposition \ref{inner cont}, we have $T$ is
approximately left $\psi_{\phi}$-amenable. Proceed a similar
arguments as in the  Example \ref{ex1}, we have a net $(a_{\alpha})$
in $A(G)$ such that $a-ba_{\alpha}\rightarrow 0$  for each $a,b\in
A(G).$ By taking $a\in A(G)$ such that $\phi(a)=1$ and $b\in
\ker\phi$, we have
$\phi(a)=\phi(a)-\phi(b)\phi(a_{\alpha})=\phi(a-ba_{\alpha})\rightarrow
0$ which is a contradiction.
\end{Example}
\begin{lemma}
Let $G$ be a locally compact group. Then $A(G)$ is approximately
left $\phi$-biprojective.
\end{lemma}
\begin{proof}
By \cite[Example 2.6]{kan} $A(G)$ is left $\phi$-amenable for each
$\phi\in\Delta(G)$. Then is approximately left $\phi$-amenable.
Applying Proposition \ref{app left amen} implies that $A(G)$ is
approximately left $\phi$-biprojective, for each
$\phi\in\Delta(A(G)).$
\end{proof}
Let $G$ be a locally compact group and let $M(G)$ be the measure
algebra with respect to $G$. It is well-known that $L^{1}(G)$ is a
closed ideal of $M(G)$. So every character of $L^{1}(G)$ has an
extension to $M(G)$, particularly the augmentation character
$\phi_{1}$. We again denote this extension by $\phi_{1}$.
\begin{Theorem}
Let $G$ be a locally compact group. $M(G)$ is approximately left
$\phi_{1}$-biprojective if and only if $G$ amenable.
\end{Theorem}
\begin{proof}
Suppose that $M(G)$ is approximately left $\phi_{1}$-biprojective.
Since $M(G)$ is unital, by Proposition \ref{inner cont} $M(G)$ is
approximately left $\phi_{1}$-amenable. Note that $L^{1}(G)$ is a
closed ideal of $M(G)$ and $\phi_{1}|_{L^{1}(G)}\neq 0$ so by
\cite[Lemma 3.1]{kan} $L^{1}(G)$ is approximately left
$\phi_{1}$-amenable. Using similar method as in the proof of Theorem
\ref{segal} $G$ is amenable.

For converse, let $G$ be an amenable group. By Johnson theorem
$L^{1}(G)$ is amenable. Hence $L^{1}(G)$ is left
$\phi_{1}$-amenable. Hence  $M(G)$ is left $\phi_{1}$-amenable. So
$M(G)$ is approximately left $\phi_{1}$-amenable. Using Proposition
\ref{app left amen} implies that $M(G)$ is approximately left
$\phi_{1}$-biprojective.
\end{proof}
\begin{cor}
Let $G$ be a locally compact group. $M(G)$ is approximately left
character biprojective if and only if $G$ is discrete and amenable.
\end{cor}
\begin{proof}
Let $G$ be a locally compact group. Suppose that  $M(G)$ is
approximately left character biprojective. Since $M(G)$ is  unital,
then by Proposition \ref{inner cont} approximately character
biprojectivity implies that $M(G)$ is approximately character
amenable. Applying \cite[Theorem 7.2]{agha} $G$ is discrete and
amenable.

For converse, let $G$ be amenable and discrete.  Then by
\cite[Proposition 4.2]{ghah pse} $M(G)$ is pseudo-amenable. Hence by
Remark \ref{rem1}, $M(G)$ is approximately character left
biprojective.
\end{proof}

 We give a Banach algebra which is not pseudo-amenable but is
approximately left $\phi$-biprojective.
\begin{Example}
Let $G$ be an infinite compact group. It is well-known that for
$p\geq 1,$ $\Delta(L^{p}(G))=\{\phi_{\rho}|\rho\in \widehat{G}\},$
where $\widehat{G}$ is the dual group of $G$ and
$\phi_{\rho}(f)=\int_{G}f(x)\overline{\rho(x)}dx,$ see \cite{hew}.
Since $G$ is compact $\widehat{G}\subseteq L^{\infty}(G)\subseteq
L^{p}(G)$. It is easy to see that $$f\rho=\phi_{\rho}(f)\rho,\quad
\phi_{\rho}(\rho)=\int_{G}\rho(x)\overline{\rho(x)}dx=\int_{G}1dx=1,\quad
(f\in L^{1}(G))$$ (we assume that $dx$ is the normalized left Haar
measure on $G$). Since $\rho \in\L^{p}(G)$, then the map $f\mapsto
f\rho$ is $w^{*}$-continuous on $L^{p}(G)^{**}$. Hence we have
$$f\rho=\tilde{\phi}_{\rho}(f)\rho,\quad \phi_{\rho}(\rho)=\tilde{\phi}_{\rho}(\rho)=1,\quad (f\in {L^{p}(G)}^{**}).$$
It means that $L^{p}(G)^{**}$ is left
$\tilde{\phi}_{\rho}$-amenable, so $L^{p}(G)^{**}$ is approximately
left $\tilde{\phi}_{\rho}$-amenable. Therefore by Proposition
\ref{app left amen}, $L^{p}(G)^{**}$ is approximately left
$\tilde{\phi}_{\rho}$-biprojective. Particularly ${L^{1}(G)}^{**}$
is approximately left $\tilde{\phi}_{\rho}$-biprojective but if
${L^{1}(G)}^{**}$ is pseudo-amenable, then by \cite[Proposition
4.2]{ghah pse} $G$ is discrete and amenable. Since $G$ is compact,
then $G$ must be finite which is a contradiction.
\end{Example}
\begin{Theorem}
Let $G$ be a locally compact $SIN$ group. $L^{1}(G)^{**}$ is
approximately left character biprojective if and only if $G$ is
amenable.
\end{Theorem}
\begin{proof}
Suppose that $L^{1}(G)^{**}$ is approximately left character
biprojective. Since $G$ is a $SIN$ group, then by the main result of
\cite{kot}, $L^{1}(G)$ has a central approximate identity. Then for
each $\phi\in\Delta(L^{1}(G))$ there exists an element $a_{0}\in
L^{1}(G)$ such that $aa_{0}=a_{0}a$ and $\phi(a_{0})=1,$ for each
$a\in L^{1}(G)$. Since for each $a\in L^{1}(G)$ two maps $b\mapsto
ab$ and $a\mapsto ba$ is $w^{*}$-continuous on $L^{1}(G)^{**}$, we
have $$aa_{0}=a_{0}a,\quad
\phi(a_{0})=\widetilde{\phi}(a_{0})=1\quad (a\in L^{1}(G)^{**}).$$
Using Proposition \ref{app left amen}, implies  that $L^{1}(G)^{**}$
is approximately left $\phi$-amenable for all
$\phi\in\Delta({L^{1}(G)}^{**})$. By \cite[Proposition 3.9]{agha}
$L^{1}(G)$ is approximately left $\phi$-amenable. Hence
\cite[Theorem 7.1]{agha} implies that $G$  is amenable.

For converse, suppose that $G$ is amenable. So by Johnson theorem,
$L^{1}(G)$ is amenable, hence $L^{1}(G)$ is left $\phi-$amenable. By
\cite[Proposition 3.4]{kan} we have $L^{1}(G)^{**}$ is left
$\tilde{\phi}-$amenable for all $\phi\in\Delta(L^{1}(G))$. Hence
$L^{1}(G)^{**}$ is approximately left $\tilde{\phi}-$amenable for
all $\phi\in\Delta(L^{1}(G))$. Now by Theorem \ref{app left amen}
$L^{1}(G)^{**}$ is approximately left character biprojective.
\end{proof}
It is well-known that for each semigroup $S$ there exists a partial
order on   $E(S)$, where $E(S)$ is the set of idempotents of $S$.
Indeed
$$s\leq t\Leftrightarrow s=st=ts\quad (s,t \in E(S)).$$

The  semigroup $S$ is called inverse semigroup, if for each $s\in S$
there exists $s^{*}\in S$ such that $ss^{*}s=s^{*}$ and
$s^{*}ss^{*}=s$ for each $s\in S.$ Inverse semigroup $S$ is called
Clifford semigroup if for each $s\in S$ there exists $s^{*}\in S$
such that $ss^{*}=s^{*}s.$ There exists a partial order on each
inverse semigroup $S$, that is,
$$s\leq t\Leftrightarrow s=ss^{*}t\quad (s,t \in S).$$
Note that these two partial orders on an inverse semigroup are the
same. Let $(S,\leq)$ be an inverse semigroup. For each $s\in S$, set
$(x]=\{y\in S|y\leq x\}$. $S$ is called uniformly locally finite if
$\sup\{|(x]|<\infty|x\in S\}$. Suppose that $S$ is an inverse
semigroup and $e\in E(S)$. $G_{e}=\{s\in S|ss^{*}=s^{*}s=e\}$ is a
maximal subgroup of $S$ with respect to $e$. For more information
 about
 semigroup theory see \cite{how}.
\begin{Theorem}
Let $S=\cup_{e\in E(S)}G_{e}$ be a Clifford semigroup such that
$E(S)$ is uniformly locally finite.  $\ell^{1}(S)$ is approximately
left character biprojective if and only if $\ell^{1}(S)$
pseudo-amenable.
\end{Theorem}
\begin{proof}
Suppose that $\ell^{1}(S)$ is approximately left character
biprojective. By \cite[Theorem
2.16]{rams},$\ell^{1}(S)\cong\ell^{1}-\oplus_{e\in
E(S)}\ell^{1}(G_{e})$. Since $\ell^{1}(G_{e})$ has a character
$\phi_{1}$(at least augmentation character), then this character
extends on $\ell^{1}(S)$ which we denote this extension with
$\phi_{1}$ again. So $\ell^{1}(S)$ is approximately left
$\phi_{1}$-biprojective. Since $\phi_{1}|_{\ell^{1}(G_{e})}\neq 0$
and  $\ell^{1}(G_{e})$ is a closed ideal of $\ell^{1}(S)$, by
Proposition \ref{ideal}, $\ell^{1}(G_{e})$ is approximately left
$\phi_{1}$-biprojective. On the other hand since $\ell^{1}(G_{e})$
is unital, by Proposition \ref{inner cont}, $\ell^{1}(G_{e})$ is
approximately left $\phi_{1}$-amenable. So by \cite[Theorem
7.1]{agha}, $G_{e}$  is amenable for all $e\in E(S)$. Thus by
\cite[Corollary 3.9]{rost} $\ell^{1}(S)$ is pseudo-amenable.

Converse is true by Remark \ref{rem1}.
\end{proof}

\section{Examples}
We give a Banach algebra  which is approximately left
$\phi$-biprojective but is not $\phi$-biprojective.
\begin{Remark}
Consider the semigroup $\mathbb{N}_{\vee}$,  with semigroup
operation $m\vee n=\max\{m,n\}$, where $m$ and $n$ are in
$\mathbb{N}$. The character space
$\Delta(\ell^{1}(\mathbb{N_{\vee}}))$ precisely consists   of the
all functions $\phi_{n}:\ell^{1}(\mathbb{N_{\vee}})\rightarrow
\mathbb{C}$ defined by
$\phi_{n}(\sum_{i=1}^{\infty}\alpha_{i}\delta_{i})=\sum_{i=1}^{n}\alpha_{i}$
for every $n\in\mathbb{N}\cup\{\infty\}$. For more information about
this semigroup algebra see \cite{dale}.   In \cite{sah1}, author
with A. Pourabbas showed that $\ell^{1}(\mathbb{N}_{\vee})$ is
$\phi_{n}$-biflat for each $n\in\mathbb{N}\cup\{\infty\}$. Since
this algebra is commutative, by \cite[Proposition 3.3]{sah1}
$\ell^{1}(\mathbb{N}_{\vee})$ is left $\phi_{n}$-amenable. Then
$\ell^{1}(\mathbb{N}_{\vee})$ is approximately left
$\phi_{n}$-amenable. By Proposition \ref{app left amen},
$\ell^{1}(\mathbb{N}_{\vee})$ is approximately character left
biprojective. Hence $\ell^{1}(\mathbb{N}_{\vee})$ is approximately
left $\phi_{\infty}$ biprojective. Moreover we showed that
$\ell^{1}(\mathbb{N}_{\vee})$ is $\phi_{n}$-biprojective for each
$n\in\mathbb{N}.$ But if $\ell^{1}(\mathbb{N}_{\vee})$ is
$\phi_{\infty}$-biprojective, then $\ell^{1}(\mathbb{N}_{\vee})$ is
character biprojective. So by \cite[Remark 3.6]{sah rom} and
\cite[Lemma 3.7]{sah rom}, the maximal ideal space of
$\ell^{1}(\mathbb{N}_{\vee})$ is finite, which is impossible because
the maximal ideal space of $\ell^{1}(\mathbb{N}_{\vee})$ is
$\mathbb{N}\cup\{\infty\}$.
\end{Remark}
We give a Banach algebra which is neither left $\phi$-amenable nor
$\phi$-biflat  but is approximately left $\phi$-biprojective. Hence
the converse of Theorem \ref{phi biflat} is not always true.
\begin{Example}
We denote $\ell^{1} $ for the set of all sequences $a=((a_{n}))$ of
complex numbers with $||a||=\sum^{\infty}_{n=1}|a_{n}|<\infty .$
Equip $\ell^{1} $ with the following product: ý
\begin{eqnarray*}ý
&\textit{(þ$a\ast
b)(n)=$}\begin{cases}a(n)b(n)\,\,\,\,\,\,\,\,\,\,\,\,\,\,\,\,\,\,\,\hspace{3.5cm}n=1\cr
ý
a(1)b(n)+b(1)a(n)+a(n)b(n)\,\,\,\,\,\,\,\,\,\,\,\,\,\,\,\,\,\,\,\,n>1,
ý\end{cases}\\ý
\end{eqnarray*}ý
and $||\cdot||$ becomes a Banach algebra. It is easy to see that
$\Delta(\ell^{1})=\{\phi_{1},\phi_{1}+\phi_{n}\}$, where
$\phi_{n}(a)=a(n)$ for each $a\in \ell^{1}$. By \cite[Example
2.9]{nem col} $\ell^{1}$ is not left $\phi_{1}$-amenable. Suppose
conversely that $A$ is $\phi$-biflat. Since $\ell^{1}$ is
commutative, by \cite[Proposition 3.3]{sah1} $\phi$-biflatness
follows that  $\ell^{1}$ is left $\phi_{1}$-amenable, which is a
contradiction.

Using \cite[Example 2.9]{nem col}, $\ell^{1}$ is approximately left
$\phi_{1}$-amenable. Then Proposition \ref{app left amen} implies
that $\ell^{1}$ is approximately left $\phi_{1}$-biprojective.
Moreover \cite[Example 2.9]{nem col} showed that $\ell^{1}$ is left
$\phi_{1}+\phi_{n}$-amenable so $\ell^{1}$ is approximately left
$\phi_{1}+\phi_{n}$-biprojective. Hence $\ell^{1}$ is approximately
left character biprojective.
\end{Example}
We give a Banach algebra which is approximately left
$\phi$-biprojective but is not approximately left $\phi$-amenable.
Then the converse of Proposition \ref{app left amen} is not always
true.
\begin{Example}
Let $S$ be a left zero semigroup with $|S|\geq 2$, that is a
semigroup with product $st=s$ for all $s,t\in S.$ For the semigroup
algebra $\ell^{1}(S)$, we have $fg=\phi_{S}(g)f$, where $\phi_{S}$
is the augmentation character on $\ell^{1}(S)$. We claim that
$\ell^{1}(S)$ is approximately left $\phi_{S}$-biprojective. To see
this, let $f_{0}\in \ell^{1}(S)$ be an element such that
$\phi_{S}(f_{0})=1.$ Define $\rho:\ell^{1}(S)\rightarrow
\ell^{1}(S)\otimes_{p}\ell^{1}(S)$ by $\rho(f)=f\otimes f_{0}$ for
all $f\in \ell^{1}(S)$. It is easy to see that $$f\cdot
\rho(g)=\rho(fg),\quad \rho(fg)=\phi_{S}(g)\rho(f),\quad
\phi_{S}\circ\pi_{A}\circ\rho(f)=\phi_{S}(f_{0}f)=\phi(f),\quad(f,g\in
\ell^{1}(S)).$$ We show that $\ell^{1}(S)$ is not approximately left
$\phi$-amenable, provided that $|S|\geq 2.$ We go toward a
contradiction and suppose that $\ell^{1}(S)$ is approximately left
$\phi$-amenable. Then there exists a net $(f_{\alpha})$ in
$\ell^{1}(S)$ such that $$\phi_{S}(f_{\alpha})=1, \quad
ff_{\alpha}-\phi_{S}(f)f_{\alpha}\rightarrow 0\quad (f\in
\ell^{1}(S)).$$ It follows that $f-\phi_{S}(f)f_{\alpha}\rightarrow
0$ for each $f\in \ell^{1}(S)$. Since $S$ has at least two elements
$s_{1},s_{2}$, take $f=\delta_{s_{1}}$ and  $f=\delta_{s_{1}}$ and
put it in $f-\phi_{S}(f)f_{\alpha}\rightarrow 0$. It follows that
$\delta_{s_{1}}=\delta_{s_{2}}$, so $s_{1}=s_{2}$ which is
impossible.
\end{Example}

\end{document}